\documentclass[12pt]{article}
\usepackage{amsmath,amsthm,amssymb,amscd}

\title{The problems of
classifying pairs of forms and local
algebras with zero cube radical are
wild\footnotetext{This is the authors' version of a work that was published in Linear Algebra Appl. 402 (2005) 135--142.}}

\author{Genrich Belitskii%
\thanks{Partially supported by
Israel Science Foundation, Grant
186/01.}\\ Dept. of Mathematics,
Ben-Gurion University of the Negev\\
Beer-Sheva 84105, Israel,
 genrich@cs.bgu.ac.il
 \and
Vitalij M. Bondarenko
\\ Institute of Mathematics,
Tereshchenkivska 3, Kiev,
Ukraine\\vit-bond@imath.kiev.ua
  \and
Ruvim Lipyanski\\ Dept. of Mathematics,
Ben-Gurion University of the Negev\\
Beer-Sheva 84105, Israel,
 lipyansk@cs.bgu.ac.il
 \and
Vladimir V. Plachotnik \\
Mech.-Math. Faculty,
Kiev National University,\\
Vladimirskaja 64, Kiev, Ukraine
 \and
Vladimir V. Sergeichuk%
\thanks{Corresponding author.
Partially supported by Ukrainian
SFFR, Grant 01.07/00132.}\\
Institute of Mathematics,
Tereshchenkivska 3, Kiev,
Ukraine\\sergeich@imath.kiev.ua}
\date{Dedicated to C. M. Ringel
on the occasion of his 60th
birthday}

\DeclareMathOperator{\rank}{rank}
\DeclareMathOperator{\Rad}{Rad}
\DeclareMathOperator{\diag}{diag}

\renewcommand{\ge}{\geqslant}

\newtheorem{theorem}{Theorem}
\newtheorem{lemma}[theorem]{Lemma}

\theoremstyle{definition}

\theoremstyle{remark}

\begin{document}
\maketitle

\begin{abstract}
We prove that over an algebraically
closed field of characteristic not
two the problems of classifying
pairs of sesquilinear forms in which
the second is Hermitian, pairs of
bilinear forms in which the second
is symmetric (skew-symmetric), and
local algebras with zero cube
radical and square radical of
dimension $2$ are hopeless since
each of them reduces to the problem
of classifying pairs of $n$-by-$n$
matrices up to simultaneous
similarity.


 \end{abstract}

%\newcommand{\rank}{\mathop{\rm rank}\nolimits}
%\DeclareMathOperator{\rank}{rank}

%%%%%%%%%%%%%%%%%%%%%
%%%%%%%%%%%%%%%%%%%%%

\section{Introduction}
\label{s0}

All matrices, vector spaces, and
algebras are considered over an
algebraically closed field $\mathbb
F$ of characteristic not two.

The problem of classifying pairs of
$n\times n$ matrices up to similarity
transformations
\begin{equation*}\label{0a}
(A,B)\longmapsto
(S^{-1}AS,\,S^{-1}BS),
\qquad\text{$S$\/ is nonsingular},
\end{equation*}
is hopeless since it contains the
problem of classifying any system of
linear operators and the problem of
classifying representations of any
finite-dimensional algebra; see
\cite{bel-ser}. Classification
problems that contain the problem of
classifying pairs of matrices up to
similarity are called \emph{wild}
and the others are called
\emph{tame}; see strict definitions
in \cite{rin}.

We prove that the problem of
classifying local algebras $\Lambda$
with $(\Rad\Lambda)^3=0$ and $\dim
(\Rad\Lambda)^2=2$ is wild (Theorem
\ref{theor}). Recall that an
\emph{algebra} $\Lambda$ over
$\mathbb F$ is a finite dimensional
vector space being also a ring with
respect to the same addition and
some multiplication such that
\begin{equation*}\label{eqy}
a (uv )=(a u)v =u (av)\qquad
\text{for all $a\in\mathbb F$ and
$u,v\in\Lambda$}.
\end{equation*}
An algebra $\Lambda$ is \emph{local}
if the set $R$ of its noninvertible
elements is closed under addition
(then $R$ is the \emph{radical} of
$\Lambda$ and is denoted by
$\Rad\Lambda$).

We prove in passing the wildness of
the problems of classifying
\begin{itemize}
  \item[(i)]
pairs of sesquilinear forms, in
which the second is Hermitian (with
respect to a nonidentity involution
on $\mathbb F$),
  \item[(ii)]
pairs of bilinear forms, in which
the second is symmetric, and
  \item[(iii)]
pairs of bilinear forms, in which
the second is skew-symmetric.
\end{itemize}
The hopeless of the problems of
classifying tuples (i)--(iii) was
also proved in \cite{ser2} by
another method (which was used in
\cite{ser1} too): each of them
reduces to the problem of
classifying representations of a
wild quiver.

Belitskii, Lipyanski, and Sergeichuk
worked on these problems when
Sergeichuk was visiting the
Ben-Gurion University of the Negev
in November and December 2003. They
discussed applications of
\cite[Theorem 4.5]{bel-ser} stating
that the problem of classifying
tensors $T\in U^*\otimes U^*\otimes
U$ on a vector space $U$ is wild
(such a tensor determines a bilinear
binary operation on $U$). Then these
authors knew that the wildness of
the problem of classifying algebras
was also proved by Bondarenko and
Plachotnik using another reduction
to a matrix problem. So we decided
to write this paper jointly.

\section{Pairs of forms}
\label{sec1}

Let $a\mapsto \bar{a}$ be any
involution on $\mathbb F$, that is,
a bijection $\mathbb F\to\mathbb F$
such that
\begin{equation*}\label{1ig}
 \overline{a+b}= \bar{a}+\bar{b},
 \quad \overline{ab}=\bar{a}
 \bar{b},\quad
 \bar{\bar{a}}=a.
\end{equation*}
For a matrix $A=[a_{ij}]$ we define
$ A^*:=\bar{A}^T =[\bar{a}_{ji}].$
If $S^*AS=B$ for a nonsingular
matrix $S$, then $A$ and $B$ are
said to be *{\it\!congruent} (the
involution $a\mapsto \bar{a}$ can be
the identity; we consider congruence
of matrices as a special case of
*congruence).

Each matrix tuple in this paper is
formed by matrices of the same size,
which is called the size of the
tuple. Denote
\[
R(A_1,\dots,A_t):=
(RA_1,\dots,RA_t), \qquad
(A_1,\dots,A_t)S:=
(A_1S,\dots,A_tS).
\]
We say that matrix tuples
$(A_1,\dots,A_t)$ and
$(B_1,\dots,B_t)$ are
\emph{equivalent} and write
\begin{equation}\label{er}
(A_1,\dots,A_t)\sim
(B_1,\dots,B_t)
\end{equation}
if there are nonsingular $R$ and $S$
such that
\[
R(A_1,\dots,A_t)S=(B_1,\dots,B_t).
\]
These tuples are *\!\emph{congruent}
if $R=S^*$.

For each $\varepsilon\in\mathbb F$,
define the pair
\begin{equation}\label{eqv1}
{\cal T}_{\varepsilon}(x,y)= \left(
\left[\begin{array}{c|c}
  \text{\Large\rm
0}&\begin{matrix}
    1&0\\0&1
  \end{matrix}\\ \hline
  \begin{matrix}
    2&1\\0&2
  \end{matrix}&\text{\Large\rm
0}
\end{array}\right],\
\left[\begin{array}{c|c}
  \text{\Large\rm
0}&\begin{matrix}
    x&0\\0&y
  \end{matrix}\\ \hline
  \begin{matrix}
    \varepsilon x^*&0\\
    0& \varepsilon y^*
  \end{matrix}&\text{\Large\rm
0}
\end{array}\right]\right)
\end{equation}
of polynomial matrices in $x,\ y,\
x^*$, and $y^*$. Then
\begin{equation}\label{eqv2}
{\cal T}_{\varepsilon}(A,B):=\left(
\left[\begin{array}{c|c}
  \text{\Large\rm
0}&\begin{matrix}
    I_n&0\\0&I_n
  \end{matrix}\\ \hline
  \begin{matrix}
    2I_n&I_n\\0&2I_n
  \end{matrix}&\text{\Large\rm
0}
\end{array}\right],\
\left[\begin{array}{c|c}
  \text{\Large\rm
0}&\begin{matrix}
    A&0\\0&B
  \end{matrix}\\ \hline
  \begin{matrix}
    \varepsilon A^*&0\\
    0& \varepsilon B^*
  \end{matrix}&\text{\Large\rm
0}
\end{array}\right]\right)
\end{equation}
for each pair $(A,B)$ of $n$-by-$n$
matrices.

The statement (a) of the following
theorem is used in the next section.

\begin{theorem}\label{th2.1}
{\rm(a)} For each
$\varepsilon\in\mathbb F$, matrix
pairs $(A,B)$ and $(C,D)$ over
$\mathbb F$ are similar if and only
if ${\cal T}_{\varepsilon}(A,B)$ and
${\cal T}_{\varepsilon}(C,D)$ are
{\rm *}\!congruent.

{\rm(b)} The problems of classifying
tuples {\rm(i)--(iii)} from Section
\ref{s0} are wild.
\end{theorem}

Define the \emph{direct sum} of
matrix tuples:
\[
(A_1,\dots,A_t)\oplus(B_1,\dots,B_t)
:=(A_1\oplus B_1,\dots,A_t\oplus
B_t).
\]
A matrix tuple is said to be
\emph{indecomposable with respect to
equivalence} if it is not equivalent
to a direct sum of matrix tuples of
smaller sizes. A tuple of square
matrices is \emph{indecomposable
with respect to {\rm*}\!congruence}
if it is not *congruent to a direct
sum of tuples of square matrices of
smaller sizes.

\begin{lemma} \label{lem2}
{\rm(a)} Each tuple of $m$-by-$n$
matrices is equivalent to a direct
sum of tuples that are
indecomposable with respect to
equivalence. This sum is determined
uniquely up to permutation of
summands and replacement of summands
by equivalent tuples.

{\rm(b)} Each tuple of $n$-by-$n$
matrices is {\rm*}\!congruent to a
direct sum of indecomposable tuples.
This sum is determined uniquely up
to permutation of summands and
replacement of summands by
{\rm*}\!congruent tuples.
\end{lemma}

\begin{proof}
(a) Each $t$-tuple of $m\times n$
matrices determines the $t$-tuple of
linear mappings ${\mathbb F}^n\to
{\mathbb F}^m$; that is, the
representation of the quiver
consisting of two vertices $1$ and
$2$ and $t$ arrows $1\longrightarrow
2$. By the Krull--Schmidt theorem
\cite[Section 8.2]{pie}, every
representation of a quiver is
isomorphic to a direct sum of
indecomposable representations,
which are determined uniquely up to
permutation and replacement by
isomorphic representations.

(b) This statement is a special case
of the following generalization of
the law of inertia for quadratic
forms \cite[Theorem 2 and
\S\,2]{ser1}: each system of linear
mappings and sesquilinear forms on
vector spaces over $\mathbb F$
decomposes into a direct sum of
indecomposable systems uniquely up
to isomorphisms of summands.
\end{proof}

\begin{proof}[Proof of Theorem
\ref{th2.1}] (a) If $(A,B)$ is
similar to $(C,D)$, then ${\cal
T}_{\varepsilon}(A,B)$ is *congruent
to ${\cal T}_{\varepsilon}(C,D)$
since $S^{-1}(A,B)S=(C,D)$ implies
\begin{equation}\label{equ10}
R^*{\cal
T}_{\varepsilon}(A,B)R={\cal
T}_{\varepsilon}(C,D),\qquad
R:=\diag
((S^*)^{-1},(S^*)^{-1},S,\,S).
\end{equation}

Conversely, let ${\cal
T}_{\varepsilon}(A,B)$ be congruent
to ${\cal T}_{\varepsilon}(C,D)$,
this means that \[R^*{\cal
T}_{\varepsilon}(A,B)R={\cal
T}_{\varepsilon}(C,D)\] for some
nonsingular $R$. Then also
\begin{gather*}
R^*{\cal
P}_{\varepsilon}(A,B)R={\cal
P}_{\varepsilon}(C,D),
   \\
{\cal P}_{\varepsilon}(x,y):= \left(
\left[\begin{array}{c|c}
  \text{\Large\rm
0}&\begin{matrix}
    1&0\\0&1
  \end{matrix}\\ \hline
  \begin{matrix}
    2&1\\0&2
  \end{matrix}&\text{\Large\rm
0}
\end{array}\right],\
\left[\begin{array}{c|c}
  \text{\Large\rm
0}&\begin{matrix}
    2&0\\1&2
  \end{matrix}\\ \hline
  \begin{matrix}
    1&0\\0&1
  \end{matrix}&\text{\Large\rm
0}
\end{array}\right],\
\left[\begin{array}{c|c}
  \text{\Large\rm
0}&\begin{matrix}
    x&0\\0&y
  \end{matrix}\\ \hline
  \begin{matrix}
    \varepsilon x^*&0\\
    0& \varepsilon y^*
  \end{matrix}&\text{\Large\rm
0}
\end{array}\right]\right).
\end{gather*}
Hence, ${\cal
P}_{\varepsilon}(A,B)\sim{\cal
P}_{\varepsilon}(C,D)$ (in the
notation \eqref{er}), and so
\begin{gather} \label{err}
{\cal F}(A,B)\oplus {\cal
G}_{\varepsilon}(A,B)\sim {\cal
F}(C,D)\oplus {\cal
G}_{\varepsilon}(C,D),
  \\  \nonumber
 {\cal
F}(x,y):= \left(
\begin{bmatrix}
    1&0\\0&1
  \end{bmatrix},\
\begin{bmatrix}
    2&0\\1&2
  \end{bmatrix},\
\begin{bmatrix}
    x&0\\0&y
  \end{bmatrix}\right),
          \\ \nonumber
  {\cal G}_{\varepsilon}(x,y):= \left(
  \begin{bmatrix}
    2&1\\0&2
  \end{bmatrix},\
  \begin{bmatrix}
    1&0\\0&1
  \end{bmatrix},\
  \begin{bmatrix}
    \varepsilon x^*&0\\
    0& \varepsilon y^*
  \end{bmatrix}\right).
\end{gather}
The equivalence
\[
{\cal G}_{\varepsilon}(C,D)\sim
\begin{bmatrix}
    2I_n&I_n\\0&2I_n
  \end{bmatrix}^{-1}\!\!\!\!
{\cal G}_{\varepsilon}(C,D)
 = \left(
\begin{bmatrix}
    I_n&0\\0&I_n
  \end{bmatrix},\
\begin{bmatrix}
    I_n/2&-I_n/4\\0&I_n/2
  \end{bmatrix},\
\dots\right)
\]
ensures that there are no triples
$\cal H$ (with matrices of size not
$0\times 0$), ${\cal H}_1$, and
${\cal H}_2$ such that
\[
{\cal F}(A,B)\sim {\cal H}\oplus
{\cal H}_1 \quad\text{and}\quad
{\cal G}_{\varepsilon}(C,D)\sim
{\cal H}\oplus {\cal H}_2.
\]
The same holds for ${\cal F}(C,D)$
and ${\cal G}_{\varepsilon}(A,B)$.
By \eqref{err} and Lemma
\ref{lem2}(a), ${\cal
F}(A,B)\sim{\cal F}(C,D)$; that is,
$R{\cal F}(A,B)={\cal F}(C,D)S$ for
some nonsingular $R$ and $S$.
Equating the corresponding matrices
of these triples gives
\[
RI_{2n}=I_{2n}S,\quad
R\begin{bmatrix}
2I_n&0\\I_n&2I_n\end{bmatrix}=
\begin{bmatrix}
2I_n&0\\I_n&2I_n\end{bmatrix}S,\quad
R\begin{bmatrix}
A&0\\0&B\end{bmatrix}=
\begin{bmatrix}
C&0\\0&D\end{bmatrix}S.
\]
By the first equality, $R=S$. By the
second equality,
\[
S=\begin{bmatrix} P&0\\Q&P\end{bmatrix}
\]
for some $P$ and $Q$. By the last
equality, $P(A,B)=(C,D)P$; that is,
$(A,B)$ is similar to $(C,D)$.
\medskip

(b) If the involution on $\mathbb F$
is not the identity and $\varepsilon
=1$, then the second matrix in
\eqref{eqv2} is Hermitian. If the
involution on $\mathbb F$ is the
identity and $\varepsilon =\pm 1$,
then the second matrix in
\eqref{eqv2} is symmetric or
skew-symmetric. This proves the
statement (b) for the pairs
{\rm(i)--(iii)} from Section
\ref{s0}.
\end{proof}

\section{Algebras} \label{s1}

An \emph{algebra} (without the
identity) is a vector space $R$ over
$\mathbb F$ with multiplication
$(u,w)\mapsto uv\in R$ being
bilinear and associative; this means
that
\begin{gather*}\label{q1}
(au+bv) w =a (uw) +b(vw),
 \qquad
u(av+bw) =a(uv)+b(uw),\\
(uv)w=u(vw)
\end{gather*}
for all $a,b\in\mathbb F$ and all
$u,v,w \in R$. Denote by $R^2$ and
$R^3$ the vector spaces spanned by
all $uv$ and, respectively, by all
$uvw$.

An algebra $\Lambda$ that contains
the identity $1$ is called
\emph{local} if the set of its
noninvertible elements is closed
under addition. Then this set is the
\emph{radical}, is denoted by $\Rad
\Lambda$, and $\Lambda/\Rad \Lambda$
is isomorphic to $\mathbb F$ (see
\cite[Section 5.2]{pie}).

\begin{theorem} \label{theor}
Let $\mathbb F$ be an algebraically
closed field of characteristic not
two.

{\rm(a)} The problem of classifying
algebras $R$ $($without the
identity$)$ over $\mathbb F$ with
$R^3=0$ and $\dim R^2=2$ is wild.

{\rm(b)} The problem of classifying
local algebras $\Lambda$ over
$\mathbb F$ with $(\Rad\Lambda)^3=0$
and $\dim (\Rad\Lambda)^2=2$ is
wild.
\end{theorem}

Due to the next lemma, the statement
(a) ensures (b).

\begin{lemma} \label{l5za}
If $R$ is an algebra from Theorem
{\rm\ref{theor}(a)}, then $R$ is the
radical of some local algebra
$\Lambda$ from Theorem
{\rm\ref{theor}(b)}, and $\Lambda$
is fully determined by $R$.
\end{lemma}

\begin{proof}
Let $R$ be an algebra for which
$R^3=0$ and $\dim R^2=2$. We
``adjoin'' the identity $1$ by
considering the algebra $\Lambda$
consisting of the formal sums
\[
a 1+u\qquad (a\in\mathbb F,\ u\in R)
\]
with the componentwise addition and
scalar multiplication and the
multiplication
\[
(a 1+u)(b 1+v)= ab 1+(a v+b u+uv).
\]
The algebra $\Lambda$ is local since
$R$ is the set of its noninvertible
elements.
\end{proof}

The next lemma reduces the problem
of classifying algebras from Theorem
\ref{theor}(a) to a matrix problem.

\begin{lemma} \label{l5z}
Every algebra $R$ for which $R^3=0$
and $\dim R^2=2$ is isomorphic to
exactly one algebra on\/ ${\mathbb
F}^{2+n}$ for some $n\ge 2$ with
multiplication
\begin{equation}\label{5.1z}
uv=\left(u^T\!\begin{bmatrix}
0_2&0\\0&A
\end{bmatrix}\!v,\
u^T\!\begin{bmatrix} 0_2&0\\0&B
\end{bmatrix}\!v,\
0,\dots,0\right)^T
\end{equation}
given by $n$-by-$n$ matrices $A$ and
$B$ that are linearly independent:
\begin{equation*}\label{5.2z}
a A+bB=0 \qquad\Longrightarrow\qquad
a=b=0.
\end{equation*}
The pair $(A,B)$ is determined by $R$
uniquely up to congruence and linear
substitutions
\begin{equation}\label{1.3z}
(A,B)\longmapsto (r_{11} A+r_{12} B,\
r_{21} A+r_{22} B),
\end{equation}
in which the matrix $[r_{ij}]$ must be
nonsingular.
\end{lemma}

\begin{proof}
Let $R$ be an algebra of dimension
$n+2$ such that $R^3=0$ and $\dim
R^2=2$. Choose a basis $e_1,e_2$ of
$R^2$ and complete it to a basis
\begin{equation}\label{bbb}
 e_1,\ e_2,\ f_1,\dots,f_{n}
\end{equation}
of $R$. Since $e_1,e_2\in R^2$ and
$R^3=0$,
\begin{equation}\label{5.vz}
e_ie_j=0,\qquad e_if_j=0,\qquad f_if_j
=a_{ij}e_1+b_{ij}e_2,
\end{equation}
in which $A=[a_{ij}]$ and
$B=[b_{ij}]$ are some $n$-by-$n$
matrices. Representing the elements
of $R$ by their coordinate vectors
with respect to the basis
\eqref{bbb} and using \eqref{5.vz},
we obtain \eqref{5.1z}. A change of
the basis $e_1,e_2$ of $R^2$ reduces
$(A,B)$ by transformations
\eqref{1.3z}. A change of the basis
vectors $f_1,\dots,f_n$ reduces
$(A,B)$ by congruence
transformations. The matrices $A$
and $B$ are linearly independent due
to \eqref{5.vz} and the condition
$\dim R^2= 2$.
\end{proof}

\begin{proof}
[Proof of Theorem \ref{theor}] Due
to Lemmas \ref{l5za} and \ref{l5z},
it suffices to prove that the
problem of classifying pairs of
matrices up to congruence and
substitutions \eqref{1.3z} is wild.
Its wildness is proved in much the
same way as \cite[Theorem
4.5]{bel-ser}.

Consider the pair
\begin{equation}\label{5.9'}
{\cal P}(x,y) :=(I_{20},0_{20})
\oplus (0_{10}, I_{10})\oplus
(I_1,I_1)\oplus {\cal T}_0(x,y)
\end{equation}
of 35-by-35 matrices, in which
${\cal T}_0(x,y)$ is defined in
\eqref{eqv1}. Let us prove that
matrix pairs $(A,B)$ and $(C,D)$ are
similar if and only if ${\cal
P}(A,B)$ reduces to ${\cal P}(C,D)$
by transformations of congruence and
substitutions \eqref{1.3z}.

If $(A,B)$ is similar to $(C,D)$,
that is, $S^{-1}(A,B)S=(C,D)$ for
some nonsingular $S$, then ${\cal
T}_0(A,B)$ is congruent to ${\cal
T}_0(C,D)$ by \eqref{equ10}, and so
${\cal P}(A,B)$ is congruent to
${\cal P}(C,D)$.

Conversely, assume that ${\cal
P}(A,B)$ reduces to ${\cal P}(C,D)$
by congruence transformations and
substitutions \eqref{1.3z}; we need
to prove that $(A,B)$ is similar to
$(C,D)$. These transformations are
independent: we can first produce
substitutions and obtain
\begin{equation}\label{eq30}
(r_{11}M_1+r_{12}M_2(A,B),\
r_{21}M_1+r_{22}M_2(A,B))
\end{equation}
where $M_1$ and $M_2(A,B)$ are the
first and the second matrices of the
pair ${\cal P}(A,B)$ and $[r_{ij}]$
is nonsingular, and then congruence
transformations and obtain
\begin{equation*}\label{eq29}
{\cal P}(C,D)=(M_1,\,M_2(C,D)).
\end{equation*}
Clearly,
\begin{align*}
\rank{(r_{11}M_1+r_{12}M_2(A,B))}=&
\rank{M_1},\\
\rank{(r_{21}M_1+r_{22}M_2(A,B))}=&
\rank{M_2(C,D)}.
\end{align*}
Since ${\cal P}(x,y)$ is defined by
\eqref{5.9'}, these equalities imply
$r_{ij}=0$ if $i\ne j$; that is,
${\cal P}(C,D)$ is congruent to
$(r_{11}M_1,\, r_{22}M_2(A,B)),$
which is congruent to
$r_{11}^{-1}(r_{11}M_1,\,
r_{22}M_2(A,B))$ because $\mathbb F$
is algebraically closed. We have
that
\begin{equation}\label{5a}
\text{${\cal P}(C,D)\ $ is congruent
to $\ (M_1,\,aM_2(A,B))$,}
\end{equation}
where $a=r_{22}/r_{11}$.

We say that a pair $P$ \emph{has a
direct summand} $D$ if $P$ is
congruent to a direct sum with the
summand $D$. By \eqref{5.9'},
$(M_1,\,M_2(A,B))$ has the direct
summand $(1,1):= (I_1,\,I_1)$, and
so $(M_1,\,aM_2(A,B))$ has the
direct summand $(1,a)$. By
\eqref{5a}, ${\cal P}(C,D)$ has the
direct summand $(1,a)$ too.

Assume that $a\ne 1$. The pair
${\cal P}(C,D)$ is congruent to a
direct sum of ${\cal T}_0(C,D)$ and
pairs of the form $(1,0)$, $(0,1)$,
and $(1,1)$. Since $a\ne 1$ and by
Lemma \ref{lem2}(b), $(1,a)$ is a
direct summand of ${\cal T}_0(C,D)$.
Then the same holds for their first
matrices; that is, $I_1$ is a direct
summand of
\begin{equation*}\label{equ11}
F:=\left[\begin{array}{c|c}
  \text{\Large\rm
0}&\begin{matrix}
    1&0\\0&1
  \end{matrix}\\ \hline
  \begin{matrix}
    2&1\\0&2
\end{matrix}&\text{\Large\rm 0}
  \end{array}\right].
\end{equation*}
This means that $S^TFS=I_1\oplus G$
for some $G$ and a nonsingular $S$.
Hence,
\[
S^T(F-F^T)S=(I_1-I_1^T)\oplus (G-G^T)=
0_1\oplus (G-G^T);
\]
this is impossible since $F-F^T$ is
nonsingular.

Hence $a=1$ and by \eqref{5a} ${\cal
P}(A,B)$ is congruent to ${\cal
P}(C,D).$ Due to \eqref{5.9'}, all
the direct summands of ${\cal
P}(A,B)$ and ${\cal P}(C,D)$
coincide except for ${\cal
T}_0(A,B)$ and ${\cal T}_0(C,D)$. By
Lemma \ref{lem2}(b), the pairs
${\cal T}_0(A,B)$ and ${\cal
T}_0(C,D)$ are congruent. By Theorem
\ref{th2.1}, $(A,B)$ is similar to
$(C,D)$.
\end{proof}

\end{document}